\documentclass[11pt]{article}
\usepackage[utf8]{inputenc}
\usepackage{graphicx, amsmath, amssymb, amsthm,anysize,float,framed,paralist,booktabs}
\usepackage{hyperref}
\hypersetup{citecolor=red, linkcolor=blue, colorlinks=true}

\newcommand{\vj}{\mathbf{j}}
\newcommand{\vJ}{\mathbf{J}}
\newcommand{\vI}{\mathbf{I}}
\newcommand{\rank}{\text{rk}}
\newcommand{\trace}{\text{tr}}

\newcommand{\scrF}{\mathcal{F}}

\newcommand{\ff}{\mathcal{F}}
\newcommand{\G}{\mathcal{G}}

\newcommand{\scrG}{\mathcal{G}}

\newcommand{\scrR}{\mathcal{R}}

\newcommand{\setC}{\mathbb{C}}

\newtheorem{theorem}{Theorem}
\newtheorem{lemma}[theorem]{Lemma}
\newtheorem{proposition}[theorem]{Proposition}

\newtheorem{remark}[theorem]{Remark}

\newtheorem{definition}[theorem]{Definition}

\newtheorem{question}{Question}

\title{Regular Intersecting Families}
\author{Ferdinand Ihringer\footnote{\hbadness=1500
Department of Mathematics: Analysis, Logic and Discrete Mathematics, Ghent University, Belgium.
The author is supported by a postdoctoral fellowship of the Research Foundation - Flanders (FWO).
Supported by ERC advanced grant 320924 while the author
was a postdoctoral fellow at the Einstein Institute of Mathematics, Hebrew University of Jerusalem, Israel.
{\tt Ferdinand.Ihringer@gmail.com}.},
Andrey Kupavskii\footnote{Moscow Institute of Physics and Technology, University of Oxford; Email:
{\tt kupavskii@yandex.ru}. Research supported by the grant RNF 16-11-10014.}
}

\begin{document}

\maketitle

\begin{abstract}
We call a family of sets  {\it intersecting}, if any two sets in the family intersect. In this paper we investigate intersecting families $\scrF$ of $k$-element subsets
  of $[n]:=\{1,\ldots, n\},$ such that every element of $[n]$ lies in the same (or approximately the same) number of members of $\scrF$.
  In particular, we show that we can guarantee $|\ff| = o({n-1\choose k-1})$ if and only if $k=o(n)$.
\end{abstract}

\section{Introduction}
Let us denote $[n]:=\{1,\ldots, n\}$ and $2^{[n]}:=\{X: X\subset [n]\}$. A {\it family} $\ff\subset 2^{[n]}$ of subsets of $[n]$ is called {\it intersecting}, if any two of its sets intersect. A family is called {\it $k$-uniform}, if it consists of sets of size $k$. For shorthand, we denote such a family by $\ff\subset {[n]\choose k}$.

The investigation of intersecting families started from the following famous result due to Erd\H{o}s, Ko and Rado.

\begin{theorem}[{\cite{ErdHos1961}}]\label{thm:ekr}
  Let $n \geq 2k$ and consider an intersecting family $\ff\subset {[n]\choose k}$. Then
  \begin{align*}
    |\ff|\le \binom{n-1}{k-1}.
  \end{align*}
  For $n > 2k$ equality occurs if and only if $\scrF$ consists of all $k$-sets
  that contain a fixed element of $[n]$.
\end{theorem}
An intersecting family with all its sets containing a fixed element of $[n]$ is called {\it trivial}.
The EKR theorem was sharpened by Hilton and Milner \cite{HM}, who determined the size of the largest {\it non-trivial} intersecting $k$-uniform family. Later, a very strong result in this direction was obtained by Frankl \cite{Fra1}. Below we give a stronger and a more convenient version of that theorem, proven in \cite{KZ}.  The \textit{degree $\delta(x)$} of $x \in [n]$ is defined as the number of members
of $\scrF$ that contain $x$. The {\it maximal degree} of $\ff$ is denoted by $\Delta(\ff)$. The {\it diversity} $\gamma(\ff)$ of $\ff$ is the number of sets from $\ff$ not containing an element of the maximal degree: $\gamma(\ff):=|\ff|-\Delta(\ff)$.

\begin{theorem}[\cite{KZ}]\label{thmKZ} Let $n>2k>0$ and $\ff\subset {[n]\choose k}$ be an intersecting family. Then, if $\gamma(\ff)\ge {n-u-1\choose n-k-1}$ for some real $3\le u\le k$, then \begin{equation}\label{eq01}|\ff|\le {n-1\choose k-1}+{n-u-1\choose n-k-1}-{n-u-1\choose k-1}.\end{equation}
\end{theorem}

Other {\it stability} results for $k$-uniform intersecting families were obtained by several researchers (cf. \cite{BBN, DF, EKL, KM, KK, Pyad, R}). Dinur and Friedgut \cite{DF} introduced the methods of analysis of Boolean functions to the study of intersecting families. Roughly speaking, they showed that {\it any} intersecting family is essentially contained in {\it juntas} with small centers. We say that a family $\mathcal J\subset 2^{[n]}$ is a {\it $j$-junta}, if there exist a subset $J\subset [n]$ of size $j$, such that the membership of a set in $\ff$ is determined only by its intersection with $J$, that is, for some family $\mathcal J^*\subset 2^{J}$ we have $\ff=\{F:F\cap J\in \mathcal J^*\}$. Here is one of their two main results.

\begin{theorem}[\cite{DF}]\label{thmdf} For any integer $r\ge 2$, there exist functions $j(r), c(r)$, such that for any integers $1 < j(r) < k < n/2$, if $\ff\subset {[n]\choose k}$ is an intersecting family with $|\ff|\ge c(r){n-r\choose k-r}$, then there exists an intersecting $j$-junta $\mathcal J$ with
$j\le j(r)$ and \begin{equation}\label{eqDF} |\ff\setminus\mathcal J|\le c(r){n-r\choose k-r}.\end{equation}
\end{theorem}
The methods of Dinur and Friedgut were developed and extended to other extremal questions on set systems by other researchers, notably Ellis, Keller, and Lifshitz \cite{EKL, KL}.\\

In this paper we study intersecting families with respect to another natural measure of non-triviality: the distribution of degrees. The question we address in this paper is as follows: how large an intersecting family may be, provided that all elements of the ground set have the same degree? We call such families {\it regular}. In what follows, we always denote the degree of an element in such a family  by $\delta$.\\

Theorem \ref{thm:ekr} implies that for $n>2k$
the largest $k$-uniform intersecting family has one element of degree $\binom{n-1}{k-1}$
while all the other elements of $[n]$ have degree $\binom{n-2}{k-2}$. Hence, the spread between the largest and the smallest degree is very large.

Using Theorem~\ref{thmKZ} with $u=3$, we get that any family of size strictly bigger than ${n-1\choose k-1}+{n-4\choose k-3}-{n-4\choose k-1}=3{n-3\choose k-2}+{n-3\choose k-3}$ has diversity strictly smaller than ${n-4\choose k-3}$. Therefore, the maximal degree of any such family is $\Delta(\ff)=|\ff|-\gamma(\ff)> 3{n-3\choose k-2}$. At the same time, the sum of the degrees of all elements except for a most popular one is $(k-1)\Delta(\ff)+k\gamma(\ff)$, and so the minimal degree of $\ff$ is at most $\frac {k-1}{n-1}\Delta(\ff)+\frac {k}{n-1}\gamma(\ff)<\frac {k-1}{n-1}\Delta(\ff)+\frac{k}{n-1}{n-4\choose k-3}<\frac {k-1}{n-1}\Delta(\ff)+{n-3\choose k-2}$. It is easy to see that the difference between the maximal and the minimal degrees is at least $\Delta (\ff)-\frac {k-1}{n-1}\Delta(\ff)-{n-3\choose k-2}=\frac{n-k}{n-1}\Delta(\ff)-{n-3\choose k-2}>\frac 3{2}{n-3\choose k-2}-{n-3\choose k-2}>0$. Therefore, we may conclude the following.

\begin{proposition}\label{prop1} If $n\ge 2k>0$, then any regular $k$-uniform intersecting family has size at most $3{n-3\choose k-2}+{n-3\choose k-3}$.
\end{proposition}


We write $S_n$ for the symmetric group on $[n]$. A \textit{symmetric family} $\scrF$ on $[n]$ is a family such that the automorphism group $Aut(\ff):=\{\sigma\in S_n: \sigma(\ff)=\ff\}$ of $\ff$ is a transitive subgroup of $S_n$, that is, for all $i,j\in [n]$ there exists a permutation $\sigma\in Aut(\ff)$ such that $\sigma(i)=j$. Clearly, any symmetric family must be regular, and the converse is not true in general.

Cameron, Frankl and Kantor \cite{Cameron1989} studied maximal intersecting families $\ff\subset 2^{[n]}$ (that is, families of size $2^{n-1}$), that are additionally symmetric or regular.
They proved the following result that relates the size of the ground set and the size of the smallest set in the family.

\begin{theorem}[{\cite[Theorem 3]{Cameron1989}}]\label{thm:bnd_on_n} Consider an intersecting family $\ff\subset 2^{[n]}$ of size $2^{n-1}$ and an arbitrary set $F\in \ff$.
  \begin{enumerate}
   \item If $\ff$ is symmetric, then $n \le |F|^2$.
   \item If $\ff$ is regular, then $n \le \frac{2}{\pi}4^{|F|}$.
  \end{enumerate}
\end{theorem}

Ellis, Kalai and Narayanan investigated $k$-uniform symmetric intersecting families
\cite{Ellis2017}. They obtained the following analogue of the theorem above.

\begin{theorem}[{\cite[Lemma 4.5]{Ellis2017}}]\label{thm:ellis_proj}
  Let $k^2-k+1$ be prime.
  A $k$-uniform symmetric intersecting family $\scrF$ on $[n]$ satisfies $n \leq k^2-k+1$.
  Further, equality holds if and only if $\scrF$ is a point-transitive projective plane of order $k-1$.
\end{theorem}

This result was already shown by Lov\'{a}sz without the characterization of equality, but not published
(it is mentioned in \cite{Lovasz1975,Lovasz1977}), and F\"{u}redi \cite{Furedi1981,Furedi1990}
for regular families. We give a second proof for this result that was obtained independently.

\begin{theorem}[{\cite[Corollary 3]{Furedi1981}}]\label{thm:char_n_max}
  A $k$-uniform regular intersecting family $\scrF$ on $[n]$ satisfies $n \leq k^2-k+1$.
  Further, equality holds if and only if $\scrF$ is a projective plane of order $k-1$.
\end{theorem}

Ellis, Kalai and Narayanan obtained the following bound on the size of a $k$-uniform symmetric intersecting family.

\begin{theorem}[{\cite[Theorem 1.3]{Ellis2017}}] \label{thm:ellis_bnd}
  There exists a constant $c > 0$ (independent of $n$ and $k$) such that for $k \leq n/2$ and any symmetric intersecting family $\ff$ we have
  \begin{align*}
     |\ff| \leq \exp\left( -\frac{c(n-2k) \log n}{k(\log n - \log k)} \right) \binom{n}{k}.
  \end{align*}
\end{theorem}
Their proof uses tools from the analysis of Boolean functions, notably, the sharp threshold result due to Friedgut and Kalai \cite{FrKa}.
They managed to show that this bound is
tight up to the constant $c$ in the exponent, if $k/n$ is bounded away from zero.

Using Theorem~\ref{thmdf}, we are able to obtain the following upper bound on the size of regular intersecting families.

\begin{theorem}\label{thmmain1} For any $r\in \mathcal N$ there exists a constant $C=C(r)$, such that for $n>Ck$ any regular family $\ff\subset {[n]\choose k}$ satisfies
$$|\ff|\le C {n-r\choose k-r}.$$
\end{theorem}

Moreover, it is possible to strengthen the theorem above to the following setting. We say that the family is {\it $\alpha$-irregular} for $\alpha\ge 1$, if the ratio between the maximal degree and the minimal degree is at most $\alpha$.

\begin{theorem}\label{thmmain2} For any $r\in \mathcal N$ and $\alpha\in \mathbb R_{\ge 1}$ there exists a constant $C=C(r,\alpha)$, such that for $n>Ck$ any $\alpha$-irregular family $\ff\subset {[n]\choose k}$ satisfies
$$|\ff|\le C {n-r\choose k-r}.$$
\end{theorem}

The next result uses algebraic methods, and may be applied for all values of $n$ and $k$. However, it works only for regular families and is weaker than Theorem~\ref{thmmain1} in the range when Theorem~\ref{thmmain1} may be applied.

\begin{theorem}\label{thm:bnd_n_small_intro}
  A regular $k$-uniform intersecting family on $[n]$ satisfies
  \begin{align*}
    |\scrF| \leq \frac{\binom{n}{k}}{1 + \frac{(n-k)(n-k-1)(n-k-2)}{k(k-1)(k-2)}}.
  \end{align*}
\end{theorem}

We only know sporadic examples for which the bound is tight. We also remark that it is not difficult to see that the bound in Theorem~\ref{thm:bnd_n_small_intro} is always stronger than the bound from Proposition~\ref{prop1}.

If $n=ck$ for some constant $c$, then all the upper bounds we give have order $\Theta({n-1\choose k-1})$. This is in sharp contrast with the situation for symmetric intersecting families, for which in \cite{Ellis2017} the upper bounds have order $o({n-1\choose k-1})$ for $n-2k=\Theta(k/\log k)$. However, this is not a shortcoming of our methods: for any $c$ we give a series of examples of a regular intersecting families in ${[n]\choose k}$, such that $n>ck$ and that have size $\Theta({n-1\choose k-1})$.
Therefore, the following is true.

\begin{proposition}\label{prop3} If $k=o(n)$, then any regular intersecting family in ${[n]\choose k}$ has size $o({n-1\choose k-1})$. At the same time, for any $c>0$ there are regular intersecting families in ${[n]\choose k}$ with $n>ck$ and which have size $\Theta_c({n-1\choose k-1})$.
\end{proposition}

We note that a similar construction was used in \cite[Theorem 3]{Furedi1981A} by F\"uredi and \cite[Theorem 2.4]{Frankl1986} by Frankl and F\"uredi. However, their studies were concerned with intersecting families with bounded maximum degree and they did not require the families to be regular. In particular, the lower bounds they prove resemble Theorem~\ref{thmmain1}, but are only meaningful for $n\gg k^2$, which is by Theorem~\ref{thm:char_n_max} not interesting as long as regular intersecting families are considered.


The structure of the paper is as follows. In Section~\ref{sec:prelim} we review some notions from the theory of association schemes, needed in the proofs of Theorems~\ref{thm:char_n_max} and \ref{thm:bnd_n_small_intro}. In Section~\ref{sec3} we prove the upper bounds on the size of regular intersecting families and some of its generalizations, in particular, to the case of degrees of subsets. In the first part of Section~\ref{sec4} we provide some general constructions of regular intersecting families, which in particular imply the corresponding part of Proposition~\ref{prop3}. In the second part of Section~\ref{sec4} we study in detail the case $n=2k$, and give some results in the case $n=2k+1$. Finally, in Section~\ref{sec5} we conclude and give future directions for research.

\section{The Johnson Scheme}\label{sec:prelim}

In the proof of Theorems~\ref{thm:char_n_max} and \ref{thm:bnd_n_small_intro} we use the theory of association schemes.
To make this paper more self-contained, this section summarizes the necessary notions and results.
We refer to Delsarte's PhD thesis \cite{Delsarte1973} as a standard reference on
the combinatorial applications of association schemes.

\begin{definition}
 Let $X$ be a finite set. A $k$-class association scheme is a pair $(X, \scrR)$,
 where $\scrR = \{ R_0, \ldots, R_k \}$ is a set of symmetric binary relations
 on $X$ with the following properties:
 \begin{enumerate}[(a)]
  \item $\{ R_0, \ldots R_k \}$ is a partition of $X \times X$.
  \item $R_0$ is the identity relation.
  \item There are constants $p_{ij}^\ell$ such that for $x, y \in X$ with $(x, y) \in R_\ell$
  there are exactly $p_{ij}^\ell$ elements $z$ with $(x, z) \in R_i$ and $(z, y) \in R_j$.
 \end{enumerate}
\end{definition}
We denote $p^0_{ii}$ by $r_i$, the \text{valency of $R_i$}. We denote $|X|$ by
$v = \sum_{i=0}^k r_i$. The relations $R_i$ can be described by their adjacency matrices
$A_i \in \setC^{v \times v}$ defined by
\begin{align*}
  (A_i)_{xy} = \begin{cases}
                1 & \text{ if } x R_i y,\\
                0 & \text{ otherwise.}
               \end{cases}
\end{align*}
Let $\vJ$ denote the all-ones matrix and let $\vj$ denote the all-ones vector.
It is easily verified that the matrices $A_i$ are Hermitian and commute pairwise,
hence we can diagonalize them simultaneously, i.e. there is a set of common eigenvectors.
From this we obtain pairwise orthogonal, idempotent Hermitian
matrices $E_j \in \setC^{v \times v}$ with the properties (possibly after reordering)
\begin{align}
  &\sum_{j=0}^k E_j = \vI, && E_0 = v^{-1} \vJ,\notag\\
  &A_i = \sum_{j=0}^k P_{ji} E_j, && E_j = v^{-1} \sum_{i=0}^k Q_{ij} A_i \label{eq:decomp_adj}
\end{align}
for some constants $P_{ij}$ and $Q_{ij}$.
From this it is clear that the
matrices $A_0, A_1, \ldots, A_k$ have $k+1$ common eigenspaces $V_0 = \langle \vj \rangle,
V_1, \ldots, V_k$, where the eigenspaces $V_j$ has dimension $f_j := \rank(E_j)=\trace(E_j)$
and $P_{ji}$ is the eigenvalue of  $A_i$ in the eigenspace $V_j$.
Note that $E_j$ is an orthogonal projection from $\setC^{v \times v}$
onto $V_j$. There are several connections between the parameters of an association
scheme, for example $r_i = P_{0i}$. For us the following is important.

\begin{lemma}[{\cite[Eq. (4.34)]{Delsarte1973}}]\label{lem:rel_P_Q}
  We have $r_i Q_{ij} = f_j P_{ji}$.
\end{lemma}

Let $\scrF$ be a subset of $X$.
Let $\chi$ be the \textit{characteristic vector} of $\scrF$, i.e.
\begin{align*}
  \chi_x = \begin{cases}
	      1 & \text{ if } x \in \scrF,\\
	      0 & \text{ otherwise. }
           \end{cases}
\end{align*}
Define the \textit{inner distribution} $a$ of $\scrF$, $|\scrF| > 0$, by
\begin{align*}
  a_i = \frac{1}{|\scrF|} \chi^T A_i \chi.
\end{align*}

The essential property of association schemes, which we use, is summarized in
the following result which is often referred to as \textit{Delsarte's linear
programming bound (LP bound)}.

\begin{theorem}[{\cite[Lemma 2.5.1 (iv) and Prop. 2.5.2]{Brouwer1989}}]\label{thm:lp_bound}
  We have $(aQ)_j = \frac{v}{|\ff|}\chi^T E_j \chi \geq 0$ with equality if
  and only if $\chi \in V_j^\perp$.
\end{theorem}
The vector $aQ$ is often referred to as the \textit{MacWilliams transform} of $a$.
A very particular, well-known case in Delsarte's LP bound is when $\chi \in \langle \vj \rangle + V_j$
for some $j \in \{ 1, \ldots, k \}$:

\begin{lemma}\label{lem:char_single_es}
Let $\chi$ be the characteristic vector of $\scrF$.
If $\chi \in \langle \vj \rangle + V_j$, then the following holds:
\begin{enumerate}[(a)]
 \item If $F \in \scrF$, then $F$ meets exactly $|\scrF| (P_{0i} - P_{ji})/ v + P_{ji}$
elements of $\scrF$ in relation $R_{i}$.
 \item If $F \notin \scrF$, then $F$ meets exactly $|\scrF| (P_{0i} - P_{ji})/ v$
elements of $\scrF$ in relation $R_{i}$.
\end{enumerate}
\end{lemma}
\begin{proof}
  Let $\psi$ be the characteristic vector of $\{ F \}$. Write
  $\chi = \frac{|\scrF|}{v} \vj + \chi'$ with $\chi' \in V_j$.
  Then the number of elements in $\scrF$ that are in relation $R_i$ to $F$ is
  \begin{align*}
    \psi^T A_i \chi &= \psi^T \left( \sum_{i=0}^k P_{ji} E_j \right) \Big( \frac{|\scrF|}{v} \vj + \chi'\Big) \\
    &= \psi^T \Big(P_{0i} \frac{|\scrF|}{v} \vj + P_{ji} (\chi - \frac{|\scrF|}{v} \vj)\Big) = |\scrF| (P_{0i} - P_{ji})/ v + P_{ji} \psi^T \chi. \qedhere
  \end{align*}
\end{proof}

The \textit{Johnson scheme} has $X = {[n]\choose k}$. Two $k$-sets are in relation $R_i$ if their intersection is of size $k-i$. We recall the eigenvalues of the Johnson scheme.
As this association scheme is \textit{cometric}, its eigenspaces have a canonical ordering which we use in the following.

\begin{lemma}[{\cite[p. 48]{Delsarte1973} and \cite[Theorem 6.5.2]{Godsil2015}}]\label{lem:evs}
  The eigenvalue of $A_i$ of the eigenspace $V_j$ is
  \begin{align*}
   P_{ji}
    &= \sum_{h=0}^i (-1)^{h} \binom{j}{h} \binom{k-j}{i-h} \binom{n-k-j}{i-h}\\
    &= \sum_{h=i}^k (-1)^{h-i+j} \binom{h}{i} \binom{n-2h}{k-h} \binom{n-h-j}{h-j}.
  \end{align*}
  In particular, $P_{jk} = (-1)^{j} \binom{n-k-j}{k-j}$.
\end{lemma}

We will use Theorem \ref{thm:lp_bound} for $j=1$ and $j=2$, so we need $Q_{i1}$
and $Q_{i2}$ explicitly.

\begin{lemma}\label{lem:dual_evs}
 We have
 \begin{enumerate}[(a)]
  \item $c_1 Q_{i1} = kn-in-k^2$, and
  \item $c_2 Q_{i2} = (k-i)(k-i-1)(n-k-i)(n-k-i-1) - 2i^2(k-i)(n-k-i) + i^2(i-1)^2$,
 \end{enumerate}
 where $c_1 = k(n-k)/f_1$ and $c_2 = k(k-1)(n-k)(n-k-1)/f_2$.
\end{lemma}
\begin{proof}
  Evaluate Lemma \ref{lem:rel_P_Q} for the stated cases using Lemma \ref{lem:evs}.
\end{proof}

The eigenspaces of the Johnson scheme have various nice combinatorial descriptions.
Let $\scrG_S$ denote the family of all $k$-sets
that contain a fixed $s$-set $S$. Clearly, $|\scrG_S| = \binom{n-s}{k-s}$. Let $\lambda_s = \binom{n-s}{k-s}/\binom{n}{k}$.
Let $\psi_S$ denote the characteristic vector of $\scrG_S$.

\begin{lemma}[{\cite[Theorem 6.3.3]{Godsil2015}}]\label{lem:char_first_es}
  Fix $s \in \{ 1, \ldots, k \}$. The set $\cup_{r=1}^s\{ \psi_R - \lambda_r \vj: R \in \binom{[n]}{r} \}$ spans $V_1+V_2+\ldots+V_s$.
\end{lemma}

We say that a family $\scrF$ is {\it $s$-subset-regular} if every $s$-set lies in the same number of elements $\delta_s$ of $\scrF$.
For any $r$ with $1 \leq r \leq s$, double counting pairs $( R, S ) \in \binom{[n]}{r} \times \binom{[n]}{s}$ with $R \subset S$
shows that every $r$-set lies in exactly $\delta_s \binom{n-r}{s-r}/\binom{k-r}{s-r}$ elements of $\scrF$.
The following is well-known, but we include a short proof for completeness.

\begin{lemma}\label{lem:char_F}
  Fix $s \geq 1$.
  Let $\scrF$ be a family of $k$-uniform sets and let $\chi$ be the characteristic vector of $\scrF$.
  Then (a) implies (b) and (c):
  \begin{enumerate}[(a)]
   \item $\ff$ is $s$-subset-regular.
   \item The vector $\chi$ is orthogonal to the eigenspaces $V_1, \ldots, V_s$.
   \item The inner distribution $a$ of $\scrF$ satisfies
    \begin{align*}
      0 = \sum_{i=0}^k Q_{i1} a_i.
    \end{align*}
  \end{enumerate}
\end{lemma}
\begin{proof}
  The eigenspaces
  $V_0$, $V_1$, \ldots, $V_k$ are pairwise orthogonal,
  so we can write $\chi = E_0 \chi + (E_1+E_2+\ldots+E_k) \chi = \frac{|\scrF|}{v} \vj + \chi'$,
  where $\chi' \in V_1 + V_2 + \ldots + V_k$.
  Hence, for any $R \in \binom{[n]}{r}$, where $1 \leq r \leq s$,
  \begin{align*}
    \chi^T \psi_R &= \left(\frac{|\scrF|}{v} \vj \right)^T \left(\lambda_r \vj\right)
		      + \chi'^T \left(\psi_R - \lambda_r \vj \right)\\
		  &= |\scrF| \lambda_r + \chi'^T \left(\psi_R - \lambda_r \vj\right).
  \end{align*}
  Since, by simple double counting, the degree of $R$ in $\ff$ is $\lambda_r |\ff|$, we have $\chi'^T \left(\psi_R - \lambda_r \vj \right) = 0$.
  Hence, $\chi$ is orthogonal to all vectors in $V_1 + V_2 + \ldots + V_s$. Hence, (a) implies (b).

  By Theorem \ref{thm:lp_bound}, $\chi \in V_1^\perp$ is equivalent to (c).
\end{proof}

\section{Upper Bounds}\label{sec3}

\subsection{Proof of Theorems~\ref{thmmain1} and \ref{thmmain2}}

Clearly, Theorem~\ref{thmmain1} is implied by Theorem~\ref{thmmain2}, so we only need to prove Theorem~\ref{thmmain2}.  In terms of Theorem~\ref{thmdf}, we put $C=\max\{2\alpha j(r), 2c(r)\}$.  Fix $n,k$, such that $n\ge Ck$ and an $\alpha$-irregular family $\ff\subset {[n]\choose k}$. Then Theorem~\ref{thmdf} states that there exists an intersecting $j$-junta $\mathcal J=\{A\in{[n]\choose k}: A\cap J\in \mathcal J^*\}$ for some $J$ with $|J|=j\le j(r)$ and $\mathcal J^*\subset 2^J$, such that $|\ff\setminus \mathcal J|\le c(r){n-r\choose k-r}$.

We may assume that $|\ff|> C{n-r\choose k-r}\ge 2 c(r){n-r\choose k-r}$, otherwise we have nothing to prove. Therefore, we have $|\ff\cap \mathcal J|/|\ff|> 1/2$.

Given the approximation by the junta $\mathcal J$, let us bound the degrees of elements in $\ff$. On the one hand, clearly, any set from $\ff\cap \mathcal J$ intersects $J$, therefore, the maximal degree of an element in $J$ is at least $|\ff\cap \mathcal J|/j$. On the other hand, the average degree of an element in $[n]$ is $\frac kn|\ff|\le|\ff|/C.$ Hence, the ratio between the maximal and the minimal degree in $\ff$ is at least $$\frac{|\ff\cap \mathcal J|/j}{|\ff|/C}>\frac {C}{2j}\ge \alpha. $$
We conclude that, if $|\ff|>C{n-r\choose k-r}$, then $\ff$ is not $\alpha$-irregular. Theorem~\ref{thmmain2} is proved. \\

\subsection{Proof of Theorem \ref{thm:bnd_n_small_intro}}

Theorem \ref{thm:bnd_n_small_intro} is implied by the following
result for $s$-subset-regular $k$-uniform intersecting families. It is a variation of Hoffman's bound.

\begin{theorem}\label{thm:bnd_n_small}
  Fix odd $s\ge 1$. An  $s$-subset-regular $k$-uniform intersecting family $\scrF$ on $[n]$ satisfies
  \begin{align*}
    |\scrF| \leq \frac{v}{1 - P_{0k}/P_{(s+2)k}} = \frac{\binom{n}{k}}{1 + \frac{\binom{n-k}{k}}{\binom{n-k-s-2}{k-s-2}}}.
  \end{align*}
  For $n > 2k$ we have equality only if the characteristic vector $\chi$ of $\scrF$
  lies in the span of $\vj$ and the eigenspace $V_{s+2}$.
\end{theorem}
\begin{proof}[Proof of Theorem \ref{thm:bnd_n_small}]
  Let $\chi$ be the characteristic vector of $\scrF$.
  First notice that
  \begin{align}
    |\scrF| = \chi^T \chi = \frac{|\scrF|^2}{v} + \sum_{j=1}^k \chi^T E_j \chi.\label{eq:Y_Ysq}
  \end{align}
  Let $A = A_k - \sum_{i=1}^s P_{ik} E_i$.
  Obviously, $A$ has the eigenvalues $P_{jk}$ for $j = 0$ and $j > s$, and the eigenvalue $0$ for the eigenspace $V_i$, where $1 \leq i \leq s$.
  From the definition of $\scrF$, $\chi^T A_k \chi = 0$.
  By Lemma \ref{lem:char_F}, $\chi^T E_i \chi = 0$ for $1 \leq 1 \leq s$.
  Hence, $\chi^T A \chi = 0$.
  Notice that, using Lemma \ref{lem:evs} and Equation \eqref{eq:decomp_adj}, $P_{(s+2)k}$ is the smallest eigenvalue of $A$, so we obtain
  \begin{align*}
    0 = \chi^T A \chi
    &= \frac{P_{0k}}{v} \chi^T \vJ \chi + \sum_{j=s+1}^k P_{jk} \chi^T E_j \chi\\
    &\geq \frac{P_{0k}}{v} \chi^T \vJ \chi + P_{(s+2)k} \sum_{j=1}^k \chi^T E_j \chi\\
    &\stackrel{\eqref{eq:Y_Ysq}}{=} \frac{P_{0k}}{v} \cdot |\scrF|^2 + P_{(s+2)k} \left( |\scrF| - \frac{|\scrF|^2}{v} \right).
  \end{align*}
  Rearranging yields
  \begin{align*}
    |\scrF| \leq \frac{v}{1 - P_{0k}/P_{(s+2)k}},
  \end{align*}
  which, together with Lemma \ref{lem:evs}, shows the first half of the assertion.
  For the second half of the assertion notice that we have equality in this bound
  only if $\chi^T E_{s+2} \chi = \sum_{j=1}^k \chi^T E_j \chi$.
\end{proof}

Evaluating Lemma \ref{lem:char_single_es} for the case of equality in Theorem \ref{thm:bnd_n_small} and $s=1$, we
obtain that a $k$-set $F \notin \scrF$ meets exactly
$$\frac{3k(k-1)(k-2)}{n^2-3kn-n+3k^2}$$
elements of $\scrF$ in $k-1$ elements. As this has to be an integer, Theorem \ref{thm:bnd_n_small}
cannot be tight for many combinations of $n$ and $k$. In particular, the bound is never
tight for $3k(k-1)(k-2) < n^2-3kn-n+3k^2$.
For odd $s \geq 1$ fixed, the same argument yields that the bound is not tight for $n$ at least $\sim \sqrt{s+2} k^{\frac{s+2}{s+1}}$.

Ray-Chaudhuri and Wilson showed in \cite[Theorem 1]{RCW1975} that a $2t$-design contains at least $\binom{n}{t}$ blocks.
Together with Theorem \ref{thm:bnd_n_small} this implies that $s$-subset-regular intersecting families do not exist
when $s$ is large compared to $k$ (vaguely $k < \frac{3}{2} (s+1)$).

\begin{remark}
  The bound of
  \begin{align*}
    \frac{v}{1 - P_{0k}/P_{(s+2)k}}
  \end{align*}
  and its consequences hold for various other association schemes, where the eigenspaces can be
  characterized in a similar way. One example for such an association scheme would be the
  $q$-analog of the Johnson scheme, the $q$-Johnson (Grassmann) scheme.
\end{remark}

\subsection{Proof of Theorem \ref{thm:char_n_max}}

For small $n$ Delsarte's linear programming (LP) bound
corresponds to Theorem \ref{thm:bnd_n_small}, but for $n$ close to $k^2-k+1$ Delsarte's LP bound is much better.
In particular, Delsarte's LP bound implies Theorem \ref{thm:char_n_max}.
In the following we prove this formally.

\begin{proof}[Proof of Theorem \ref{thm:char_n_max}]
  The inner distribution $a$ of an intersecting family $\scrF$ has the form
  $a = (1, a_1, \ldots, a_{k-1}, 0)$.
  In terms of Lemma~\ref{lem:dual_evs}, denote $\alpha_i:=c_2 Q_{i2}$
  and $\beta_i:=c_1Q_{i1}$.
  By Lemma \ref{lem:char_F}, we know that $\sum_{i=0}^{k-1} \beta_i a_i = 0$.
  By Theorem \ref{thm:lp_bound}, we know that $\sum_{i=0}^{k-1} \alpha_i a_i \geq 0$.
  Hence, $\sum_{i=0}^{k-1} (\alpha_i - (k-1)(n-k-1)\beta_i)a_i \geq 0$.
  Using Lemma~\ref{lem:dual_evs}, we have that
  \begin{align*}
    \gamma_i := \alpha_i - (k-1)(n-k-1)\beta_i = -i(n-2)(kn-in-k^2+i).
  \end{align*}
  If $n > k^2-k+1$, then $\gamma_i < 0$ for all $i \in \{ 1, \ldots, k-1\}$.
  Hence, $\sum_{i=0}^{k-1} \gamma_i a_i \geq 0$ implies $a = (1, 0, \ldots, 0)$, which is a contradiction.
  Thus, we can assume $n = k^2-k+1$. Then $\gamma_i = 0$ for $i \in \{ 0, k-1 \}$
  and $\gamma_i < 0$ for $i \in \{ 1, \ldots, k-2 \}$.
  Hence, $\sum_{i=0}^{k-1} \gamma_i a_i \geq 0$ implies $a = (1, 0, \ldots, 0, a_{k-1}, 0)$.
  Now $\sum_{i=0}^{k-1} \beta_i a_i = 0$ together with Lemma~\ref{lem:char_F} implies $(k-1)k = a_{k-1}$.
  Therefore, $|\scrF| = k^2-k+1$, each element in $[k^2-k+1]$ has degree $k$, and each set in $\scrF$ meets all other sets in $\scrF$ in exactly one point.
  This is a well-known definition of a projective plane of order $k-1$.
\end{proof}

A similar method also gives us the following lower bound on the size of a
regular $k$-uniform intersecting family $\ff$.

\begin{lemma}\label{lem:lower_bnd}
  If a regular $k$-uniform intersecting $\scrF$ exists, then
  \begin{align*}
   |\scrF| \geq 1 + \frac{k(n-k)}{k^2-n}
  \end{align*}
  with equality only if $\sum_{i=1}^{k-2} a_i = 0$.
\end{lemma}
\begin{proof}
  Let $a$ be the inner distribution of $\scrF$.
  By Lemma \ref{lem:char_first_es}, $kn-k^2 + \sum_{i=1}^{k-1} (kn-in-k^2)a_i =0$.
  It is easy to see that under this constraint $|\scrF| = \sum_{i=0}^{k-1} a_i$
  is minimized if $a = (1, 0, \ldots, 0, a_{k-1}, 0)$.
\end{proof}

\section{Lower Bounds}\label{sec4}

We have already shown that there are no regular families for $n>k^2-k+1$ and that for $n=k^2-k+1$  a projective plane of order $k-1$ is the only possible example. We note  that the existence of projective planes is only known for $k-1$ being a prime power. In the first part of this section we give some general constructions of how to build bigger regular intersecting families out of smaller ones, and derive some general lower bounds. In the second part of this section we discuss the case $n=2k$ and $n=2k+1$, as well as some results for small values of $k$.

We note that Ellis, Kalai, and Lifshitz in \cite[Section 4]{EKL} construct symmetric (and thus regular) $k$-uniform intersecting families on $[n]$ for all $k$ and $n$ satisfying $k\ge 1.1527\sqrt{n}$. They rely on the relation between intersecting families and difference covers for $\mathbb Z_n$. We could not improve their results for the case of regular intersecting families, so a very interesting question that remains is whether regular intersecting families exist for all sufficiently large $n,k$ satisfying $k\ge (1+o(1))\sqrt{n}$.

\subsection{General Constructions and Proposition~\ref{prop3}}
For any regular family $\ff\subset {[n]\choose k}$ we put $\alpha(\ff):=k/n$. We call this parameter the {\it ratio} of $\ff$.
Due to the simple equality $|\ff|k = n \delta$, $\alpha(\ff) = \delta/|\ff|$ and $\alpha(\ff)$ tells us, what proportion of all sets from $\ff$ contains a given element.

Let us start with giving some general ways of constructing a regular intersecting family.
Assume that we have a regular intersecting family $\ff\subset{X\choose k}$ on the set $X$, a regular intersecting family $\ff_2\subset {Y\choose k_2}$ on the set $Y$, and a regular (but not necessarily intersecting) family $\G\subset {Z\choose m}$ on the set $Z$. Assume additionally that the sets $X,Y,Z$ are pairwise disjoint and that the ratios of $\ff$ and $\G$ are the same: $\alpha(\ff) = \alpha(\G)$.

Fix a parameter $l$, $1\le l< \min\{|F_1\setminus F_2|:F_1,F_2\in \ff, F_1\ne F_2\}$. Then the following families are intersecting and regular:
\begin{align}
  &\ff^l:=\Big\{F'\in {X\choose k+l}: F'\supset F\text{ for some }F\in \ff\Big\}, && \ff^l\subset{X\choose k+l};\\
\label{eq11} &\ff+\G :=\big\{F\cup G: F\in \ff, G\in \G\big\}, &&\ff+\G\subset {X\cup Z\choose k+m};\\
&\ff\times \ff_2:=\big\{F\times F_2:F\in \ff, F_2\in \ff_2\big\}, &&\ff\times\ff_2\subset {X\times Y\choose kk_2}.
\end{align}
It is easy to see that each of the families above is intersecting. Let us show that all these families are regular. It is obvious in the case of $\ff\times \ff_2$.

This is the most difficult to show in the case of family $\ff^l$. First, note that, due to the choice of $l$, for any $F'\in \ff^l$ there is a {\it unique} set $F\in \ff$, such that $F'\supset F$. Thus, we may partition $\ff^l$ into families $\ff^l(F):=\{F'\in {X\choose k+l}: F'\supset F\}, $ where $F\in \ff$. The elements in $\ff^l(F)$ have two possible degrees. Each element $x\in F$ has degree $|\ff^l(F)| = {|X|-k\choose l}$, while any element $x\in X\setminus F$ has degree ${|X|-k-1\choose l-1}$. It is clear that, since $\ff$ is regular, each $x\in X$ has degree of the first type in $\alpha(\ff)|\ff|$ families $\ff^l(F)$ and the degree of the second type in all the others. Since the degree of each element in $\ff^l$ is the sum of its degrees in $\ff^l(F)$, we conclude that $\ff'$ is regular. We note that $|\ff^l| = {|X|-k\choose l}|\ff|$.

The family $\ff+\G$ is regular, since each element from $X$ is contained in $\alpha(\ff)$-fraction of all sets in $\ff+\G$, and each element from $Y$ is contained in $\alpha(\G)$-fraction of all sets in $\ff+\G$. However, $\alpha(\ff) = \alpha(\G)$ by assumption. We note that $|\ff+\G| = |\ff||\G|$ and that $\alpha(\ff+\G) = \alpha(\ff)$.\\

We use Construction \eqref{eq11} to show the lower bound from Proposition~\ref{prop3}. Fix $c>0$ from the proposition. Recall that we need to construct a regular intersecting family with ratio at most $1/c$ and size $\Theta({n-1\choose k-1})$. In particular, $\binom{n}{k}$ and $\binom{n-1}{k-1}$ are of the same order of magnitude. Find a prime power $q$, such that $(q^2+q+1)/(q+1)>c$. Put $\ff$ to be the projective plane of order $q$ on the set $X$. Note that $\ff$ is regular, intersecting, and $\alpha(\ff) = (q+1)/(q^2+q+1)$. Take $\G={Z\choose l(q+1)}$, where $Z$ is a set of size $l(q^2+q+1)$ disjoint with $X$. Here we think of $q$ that is fixed and $l$ that tends to infinity. Then $\alpha(\G)=\alpha(\ff)$. Thus, the family $\ff+\G$ is regular, intersecting, and satisfies $\alpha(\ff+\G)<1/c$. Moreover, $\ff+\G\subset  {|X\cup Z|\choose (l+1)(q+1)}$ and
$$\frac{|\ff+\G|}{{|X\cup Z|\choose (l+1)(q+1)}} = \frac{(q^2+q+1){l(q^2+q+1)\choose l(q+1)}}{{(l+1)(q^2+q+1)\choose (l+1)(q+1)}}>\frac{(q^2+q+1){l(q^2+q+1)\choose l(q+1)}}{{l(q^2+q+1)\choose l(q+1)}\frac{\big((l+1)(q^2+q+1)\big)^{q^2+q+1}}{\big(lq^2\big)^{q^2} \big(l(q+1)\big)^{q+1}}}=\Theta_q(1).$$

\subsection{Case $n=2k$}

In the case $n=2k$ no intersecting family of $k$-sets can have size bigger than  $\binom{2k-1}{k}$.
Brace and Daykin \cite{Brace1972} (in somewhat different terms) showed that for $k$ not a power of $2$, this bound is tight.

\begin{theorem}[{\cite[Example 1]{Brace1972}}]\label{thm25}
  If $k$ is not a power of $2$, then a largest $k$-uniform regular intersecting
  family on $[2k]$ has size $\binom{2k-1}{k}$.
\end{theorem}

Brace and Daykin also noted that it is easy to see that if $k$ is a power of $2$
that then a regular family of size $\binom{2k-1}{k}$ would have a non-integer degree,
that is there the bound of $\binom{2k-1}{k}$ is not tight. We complete their result
by showing the following.

\begin{theorem}\label{thm:neq2k_k_power_of_2}
  If $k\ge 4$ is a power of $2$, then the largest $k$-uniform regular intersecting
  family on $[2k]$ has size $\binom{2k-1}{k}-3$.
\end{theorem}
Note that there is no regular intersecting family for $k=2$, $n=4$, and substituting $k=2$ in the bound above gives $0$. We also remark that from the proof we give below it is fairly easy to reconstruct the proof of Theorem~\ref{thm25}.
\begin{proof} First we show the upper bound. Let $\ff$ be the largest regular intersecting family for $n=2k$, $k=2^t$. Consider the family $\ff(2k):=\{F\in\ff:2k\in F\}$. Simple counting shows that $|\ff(2k)| = \frac{1}{2} |\ff|$. Assume that the degree of an element $i$ in $\ff(2k)$ is $d(i)$. If all $d(i)$, $i=1,\ldots, 2k-1$, are the same, then $\frac 12|\ff|=|\ff(2k)| = d(1)\frac{2k-1}{k-1}$, and thus $2(2k-1)$ must divide $|\ff|$ and $|\ff|\le {2k-1\choose k}-(2k-1)$. (Here we use an easy-to-check fact that $(2k-1)$ divides ${2k-1\choose k}$, but $2$ does not.)

Otherwise, consider the family \begin{equation}\label{eq38}\bar\ff(2k):={[2k-1]\choose k}\setminus \big\{[2k] \setminus F:F\in \ff(2k)\big\}.\end{equation}
Clearly, $\ff\setminus \ff(2k)\subset \bar\ff(2k)$. Moreover, the degree of the element $i\in[2k-1]$ in $\bar \ff(2k)$ is $\bar d(i) = {2k-2\choose k-1}-|\ff(2k)|+d(i)$. Therefore, the degree of $i\in[2k-1]$ in $\ff':=\ff(2k)\cup \bar\ff(2k)$ is \begin{equation}\label{eq39}{2k-2\choose k-1}-|\ff(2k)|+2d(i).\end{equation}

We have $|\ff| = {2k-1\choose k}-(2l+1)$ for some $l\ge 0$. Thus  $\ff\setminus \ff(2k)=\bar\ff(2k)\setminus \mathcal G$, where $|\mathcal G|=2l+1$.  Since not all $d(i)$, $i=1,\ldots, 2k-1$, are equal to each other, by the displayed formula above two degrees in $\ff'$ differ by at least $2$. Thus, one has to delete at least two sets from $\ff'$ to get a regular family, which implies that $l\ge 1$ and proves $|\ff|\le {2k-1\choose k}-3$.
This shows the upper bound in Theorem \ref{thm:neq2k_k_power_of_2}.\\


Now let us show the lower bound. The following lemma is the key step. (We postpone the proof of the lemma until the end of the proof of the theorem.)

\begin{lemma}\label{lem666} There exists a family $\mathcal Q\subset{[2k-1]\choose k-1}$ of size $\frac 12({2k-1\choose k-1}-3)$ with the degrees $d'(i)$ of $i\in [2k-1]$ satisfying $d'(1)-1 = d'(2)-1 = \ldots = d'(3k/2)-1 = d'(3k/2+1) = \ldots =  d'(2k-1)$ and such that the sets $A_1:=[k+1,2k-1]$, $A_2:=[1,k/2]\cup [3k/2+1,2k-1]$, and $A_3:=[k/2+1,k]\cup [3k/2+1,2k-1]$ are not in $\mathcal Q$.
\end{lemma}

Double counting the sum of $d'(i)$ in such $\mathcal Q$, we get that the degree of, say, $2k-1$ equals \begin{equation}\label{eq40}\frac{1}{2k-1}\Big(\frac {k-1}2\Big[{2k-1\choose k-1}-3\Big]-\frac {3k}2\Big)=\frac{k-1}{2(2k-1)}{2k-1\choose k-1}-\frac 32 = \frac 12{2k-2\choose k-2}-\frac 32,\end{equation}
which, in particular, is an integer number (recall that ${2k-1\choose k-1}$ is odd, and so is ${2k-2\choose k-2} = \frac {k-1}{2k-1}{2k-1\choose k-1}$). This shows that there are no divisibility obstructions to the existence of $\mathcal Q$.

 We put $\ff(2k):=\{\{2k\}\cup Q: Q\in \mathcal Q\}$. Next, we consider the family $\bar\ff(2k)$ defined as in \eqref{eq38}. The family $\ff':=\ff(2k)\cup\bar\ff(2k)$ has size ${2k-1\choose k}$ and the degrees of the first $3k/2$ elements are bigger by two than the degrees of the last $k/2$ elements as we will see in the following. The degree of $2k$ in $\ff'$ is $|\ff(2k)| = \frac 12({2k-1\choose k}-3)$, which is the same as the degree of, say, $2k-1$, equal by \eqref{eq39} and \eqref{eq40} to
$${2k-2\choose k-1}-\frac 12\left({2k-1\choose k-1}-3\right)+{2k-2\choose k-2}-3 = \frac 12\left({2k-1\choose k-1}-3\right).$$
Moreover, $\ff'$ contains $[2k-1]\setminus A_i$ for $i=1,2,3$. The family $\{[2k-1]\setminus A_i:i = 1, 2, 3\}$ has degree $2$ on the first $3k/2$ elements and $0$ on the remaining ones. Thus, deleting these sets from $\ff'$, we get a regular family $\ff$ of size ${2k-1\choose k}-3$.

It remains to show that $\ff'$ is intersecting. Clearly, $\ff(2k)$ is intersecting as all
elements contain $2k$ and $\bar\ff(2k)$ is intersecting as all its elements are in $\binom{[2k-1]}{k}$.
By the definition of $\bar\ff(2k)$ in \eqref{eq38}, if $F \in \ff(2k)$, then $[2k] \setminus F \notin \bar\ff(2k)$.
Hence, each element of $\ff(2k)$ intersects all elements of $\bar\ff(2k)$.
\end{proof}

\begin{proof}[Proof of Lemma~\ref{lem666}] Due to \eqref{eq40}, we only need to show that there exists $\mathcal Q$ of desired size and without $A_i$, in which any two degrees differ by at most $1$.

We start with a construction of some auxiliary family $\mathcal Q'$ which satisfies $|\mathcal Q'| = |\mathcal Q|$. The first part of $\mathcal Q'$ is the family $\mathcal Q(2k-1)$ of all sets containing the element $2k-1$ except for the sets $A_i$. This accounts for ${2k-2\choose k-2}-3$ sets, and we have to choose $\frac 12({2k-1\choose k-1}-3)-{2k-2\choose k-2}+3 = \frac 1{2k-2}{2k-2\choose k-2}+\frac 32$ sets more. The remaining sets we choose out of ${[2k-2]\choose k-1}$ in such a way that the degrees of
the elements in $[2k-2]$ differ by at most one and are non-decreasing as $i\in[2k-2]$ increases. This could be done as follows (the same idea was used in the proof of Theorem 1 in \cite{Brace1972}).

We start with any family $\mathcal Q''$ in ${[2k-2]\choose k-1}$ of needed size. Then, if there are $g,h\in [2k-1]$, such that the degree of $g$ is bigger than the degree of $h$ by at least $2$, then we take a set $A$, such that $g\in A$, $h\notin A$, and $A':=\{h\}\cup A\setminus \{g\}$ is not in the family, and replace $A$ with $A'$. We call such an operation a {\it $(g,h)$-replacement}. Repeating this procedure will eventually lead to a family $\mathcal Q(\overline{2k-1})$, in which any two degrees differ by at most $1$.
We may also w.l.o.g. assume that elements are ordered from the ones with the smaller degree to the ones with the larger degree.

The resulting family is $\mathcal Q':=\mathcal Q(\overline{2k-1})\cup \mathcal Q(2k-1)$. Due to the fact that $A_i\notin \mathcal Q(2k-1)$, the degrees of the first $3k/2$ elements are bigger by 2 than the degrees of $3k/2+1,\ldots, 2k-2$. However, the degree of any $i\in [3k/2]$ in $\mathcal Q'$ is at most
$${2k-3\choose k-3}-1+\Big\lceil\frac 12\big(\frac 1{2k-2}{2k-2\choose k-2}+\frac 32\big)\Big\rceil\le\frac{k-\frac32}{2k-2}{2k-2\choose k-2}+\frac 12.$$
where the first term in the left hand side is the degree of $1$ in $\mathcal Q(2k-1)$ and the second term is the upper integer part of the average degree in $\mathcal Q(\overline{2k-1})$. The right hand side is at most the right hand side of \eqref{eq40} plus $1$ (which should be the degree of $i\in[3k/2]$ in $\mathcal Q$), since the difference is $\frac 1{4k-4}{2k-2\choose k-2}-\frac 54\ge 0$ for any $k\ge 4$. Similarly, the degree of each element in $[3k/2+1,2k-2]$ in $\mathcal Q'$ is at most what it should be in $\mathcal Q$.

This implies that we can obtain the family $\mathcal Q$ with degrees that differ by at most $1$ doing $(2k-1,h)$-replacements (and thus transferring the excess of the degree of $2k-1$ to other elements). In particular, since $2k-1\in A_i$, it implies that such $\mathcal Q$ will not contain $A_i$, $i=1,2,3$. The lemma is proven.
\end{proof}

\subsection{The cases $n=2k+1$ and $k=4$}

Consider the case $n = 2k+1$. The bound in Theorem \ref{thm:bnd_n_small} in this case is $\frac{k-2}{2k-1} \binom{2k+1}{k}$.
For $k=3$ this is $7$ and Theorem \ref{thm:char_n_max} states that this is only obtained
if $\scrF$ is the Fano plane. It is easy to see that the bound is always an integer
and that the degree in case of equality is the integer $\frac{k-2}{2k-1} \cdot \frac{k}{2k+1} \binom{2k+1}{k} = \frac{k-2}{2k-1} \binom{2k}{k-1}$.
We found examples that reach this bound for $k \in \{3, 4\}$. For $k=3$ there is one example, the Fano plane.
By computer, we classified all examples for $k=4$ where there are exactly two regular examples of size $36$ up to isomorphism.
The following table summarizes our knowledge for small $k$ and $n = 2k+1$.

\begin{center}
\begin{tabular}{lrrrrr}
 $k$ & 3 & 4 & 5 & 6 \\ \midrule
 Theorem \ref{thm:bnd_n_small} & $7$ & $36$ & $154$ & $624$  \\ \midrule
 Largest & $7$ & $36$ & $\geq 110$ & $\geq 442$ \\ \midrule
 $\delta$ & $3$ & $16$ & $\geq 50$ & $\geq 204$ \\
\end{tabular}
\end{center}

As we know the complete situation for $k=3$, here is a table of computer results
for $k=4$. The general bound refers to either Theorem \ref{thm:bnd_n_small}
(including improvements due to the fact that the size of the regular intersecting family
and its degree are integers), or to Theorem \ref{thm:char_n_max}.

\begin{center}
\begin{tabular}{lrrrrrr}
 $n$ & $8$ & $9$ & $10$ & $11$ & $12$ & $13$ \\ \midrule
 General Bound & $34$ & $36$ & $35$ & $33$ & $33$ & $13$ \\ \midrule
 Largest & $32$ & $36$ & $20$ & $11$ & $\geq 12$ & $13$ \\ \midrule
 $\delta$ & $16$ & $16$ & $8$ & $4$ & $\geq 4$ & $4$ \\
\end{tabular}
\end{center}
%

\section{Conclusion}\label{sec5}

All known finite projective planes have a prime power as order.
It is a famous and long-standing open problem to decide whether there exist projective planes
which do not have prime power order. In light of this it is clear that determining
the existence of regular $k$-uniform intersecting families is hard for $n=k^2-k+1$.

There are several other problems for which it seems to be more feasible to
obtain new results. We list some of them below.

The Bruck-Ryser-Chowla theorem \cite{Bruck1949} implies that if there is a projective
plane of order $k-1$ and $k-1$ is congruent $1$ or $2$ modulo $4$, then $k-1$
is the sum of two squares. This implies that Theorem \ref{thm:bnd_on_n} is
not tight for the orders $6, 14, 21, 22, \ldots$. The case of order $10$ was ruled
out separately by computer \cite{Lam1989}.

\begin{question}
  For those $k$ for which the non-existence of a projective plane of order $k-1$
  is known, what is the largest $n$ for which a regular $k$-uniform intersecting
  family exists?
\end{question}

\begin{question}
  Are there examples for which the bound in Theorem \ref{thm:bnd_n_small}
  with $n \geq 2k+1$ and $s=1$ is tight except for $(n,k) = (7,3)$ and $(n,k) = (9,4)$?\footnote{In an earlier version of this document we mistakenly did not specify $s=1$. For general $s$ Adam S. Wagner noticed that there exists a $3$-regular intersecting family which reaches the bound in Theorem \ref{thm:bnd_n_small}, see \cite{Wagner2019}.}
\end{question}

Theorem \ref{thmdf} is based on the analysis of Boolean functions in the Hamming
graph, that means an investigation on how $0$-$1$-vectors can lie in certain
eigenspaces of the adjacency matrix of the Hamming graph. It would be very interesting
to do this investigation directly in the Johnson scheme with more direct algebraic
techniques to obtain a non-asymptotic version of Theorem~\ref{thmmain1}.

\section*{Acknowledgements}
The research of the first author is supported by ERC advanced grant 320924
and he is supported by a postdoctoral fellowship of the Research Foundation - Flanders (FWO).
The research of the second author is supported the grant RNF 16-11-10014.


\begin{thebibliography}{99}
\bibitem{BBN} B. Bollob\'as, B.P. Narayanan and A.M. Raigorodskii.
\newblock On the stability of the Erd\H os--Ko--Rado theorem.
\newblock {\em J. Comb. Th. Ser. A} 137 (2016), 64--78.

\bibitem{Brace1972}
A. Brace and D. E.  Daykin.
\newblock Sperner-type theorems for finite sets.
\newblock In: {\em Combinatorics}, D. R. Woodall and D. J. A. Welsh, eds,
18--37, Inst. Maths. Applics., Southend-on-Sea, 1972.

\bibitem{Brouwer1989}
A.~E. Brouwer, A.~M. Cohen, and A.~Neumaier.
\newblock {\em Distance-regular graphs}, volume~18 of {\em Ergebnisse der
  Mathematik und ihrer Grenzgebiete (3) [Results in Mathematics and Related
  Areas (3)]}.
\newblock Springer-Verlag, Berlin, 1989.

\bibitem{Bruck1949}
R.~H. Bruck, and H.~J. Ryser.
\newblock The nonexistence of certain finite projective planes.
\newblock {\em Canad. J. Math.}, 1:88--93, 1949.

\bibitem{Cameron1989}
P.~J. Cameron, P.~Frankl, and W.~M. Kantor.
\newblock Intersecting families of finite sets and fixed-point-free
  $2$-elements.
\newblock {\em Europ. J. Combin.}, 10:149--160, 1989.

\bibitem{Delsarte1973}
P.~Delsarte.
\newblock An algebraic approach to the association schemes of coding theory.
\newblock {\em Philips Res. Rep. Suppl.}, (10):vi+97, 1973.

\bibitem{DF} I. Dinur, E. Friedgut.
\newblock Intersecting families are essentially contained in juntas.
\newblock {\em Combin. Probab. Comput.} 18 (2009), 107--122.

\bibitem{Ellis2017}
D.~Ellis, G.~Kalai, and B.~Narayanan.
\newblock On symmetric intersecting families.
\newblock {\em arXiv:1302.3636v3 [math.CO]}.

\bibitem{EKL} D. Ellis, N. Keller, N. Lifshitz.
\newblock Stability versions of Erd\H{o}s-Ko-Rado type theorems, via isoperimetry.
\newblock {\em J. Eur. Math. Soc.}, to appear.


\bibitem{ErdHos1961}
P.~Erd{\H{o}}s, C.~Ko, and R.~Rado.
\newblock Intersection theorems for systems of finite sets.
\newblock {\em Quart. J. Math. Oxford Ser. (2)}, 12:313--320, 1961.


\bibitem{Fra1} P. Frankl.
\newblock Erdos-Ko-Rado theorem with conditions on the maximal degree.
\newblock {\em J. Combin. Theory Ser. A} 46 (1987), N2, 252--263.

\bibitem{Frankl1986} P. Frankl, Z. F\"uredi.
\newblock Finite projective spaces and intersecting hypergraphs.
\newblock {\em Combinatorica} {\bf 6(4)} (1986), 335--354.

\bibitem{FrKa} E. Friedgut and G. Kalai.
\newblock Every monotone graph property has a sharp threshold.
\newblock {\em Proc. Amer. Math. Soc.} 124 (1996), 2993--3002.

\bibitem{Furedi1981} Z. F\"uredi.
\newblock Maximum degree and fractional matchings in uniform hypergraphs.
\newblock {\it Combinatorica} {\bf 1} (1981), 155--162.

\bibitem{Furedi1981A} Z. F\"uredi.
\newblock Erd{\H{o}}s-Ko-Rado Type Theorems with upper bounds on the maximum degree.
\newblock In: {\em Algebraic Methods in graph theory,} (Szeged, Hungary, 1978), L. Lov\'{a}s et al. eds.
\newblock {\it Proc. Colloq. Math. Soc. J. Bolyai} {\bf 25}, North-Holland, Amsterdam 1981, 177--207.

\bibitem{Furedi1990} Z. F\"uredi.
\newblock Covering pairs by $q^2 + q + 1$ sets.
\newblock  {\it J. Combin. Theory, Ser.~A}  {\bf 54} (1990), 248--271.

\bibitem{Godsil2015}
C.~Godsil and K.~Meagher.
\newblock {\em {E}rd{\H o}s-{K}o-{R}ado Theorems: Algebraic Approaches}.
\newblock Number 149 in Cambridge Studies in Advanced Mathematics. Cambridge
  Univ. Press, December 2016.

\bibitem{HM} A.J.W. Hilton, E.C. Milner.
\newblock Some intersection theorems for systems of finite sets.
\newblock {\em Quart. J. Math. Oxford} 18 (1967), 369--384.

\bibitem{KM} P. Keevash, D. Mubayi. 
\newblock Set systems without a simplex or a cluster.
\newblock {\em Combinatorica}, 30(2) (2010), 175--200.

\bibitem{KL} N. Keller, N. Lifshitz.
\newblock The Junta Method for Hypergraphs and Chv\'atal's Conjecture.
\newblock {\em Electron. Notes Discrete Math.}, 61 (2017), 711--717.

\bibitem{KK} S. Kiselev, A. Kupavskii.
\newblock  Sharp bounds for the chromatic number of random Kneser graphs and hypergraphs
\newblock arXiv:1810.01161

\bibitem{KZ} A. Kupavskii, D. Zakharov.
\newblock Regular bipartite graphs and intersecting families.
\newblock {\em J. Comb. Theory Ser. A} 155 (2018), 180--189.

\bibitem{Lam1989}
C.~W.~H. Lam, L.~Thiel, and S.~Swiercz.
\newblock The nonexistence of finite projective planes of order {$10$}.	
\newblock {\em Canad. J. Math.}, 41(6):1117--1123, 1989.

\bibitem{Lovasz1975}
L. Lov\'{a}sz.
\newblock On minmax theorems of combinatorics (in Hungarian).
\newblock {\em Matematikai Lapok}, 26 (1975), 209--264.

\bibitem{Lovasz1977}
L. Lov\'{a}sz.
\newblock Doctoral thesis (in Hungarian), Szeged 1977.

\bibitem{Pyad} M. Pyaderkin. 
\newblock On the stability of some Erdos-Ko-Rado type results
\newblock {\em Discrete Math.}, 340(4) (2017), 822--831.

\bibitem{R} A. M. Raigorodskii.
\newblock On the stability of the independence number of a random subgraph.
\newblock {\em Dokl. Math.} 96(3) (2017), 628--630.

\bibitem{RCW1975}
D. K. Ray-Chaudhuri, R. M. Wilson.
\newblock On $t$-designs.
\newblock {\em Osaka Journal of Mathematics}, 12(3) (1975), 737--744.

\bibitem{Wagner2019} A. Z. Wagner.
\newblock Refuting conjectures in extremal combinatorics via linear programming.
\newblock {\em arXiv:1903.05495}.

\end{thebibliography}

\end{document}